\documentclass[reqno,10pt]{amsart}
\usepackage{amssymb,array,bm,tikz,rotating,mathdots,hyperref,mathrsfs}
\usepackage[vcentermath]{youngtab}
\usetikzlibrary{matrix}

\newcommand{\BB}{\mathcal{B}}
\newcommand{\CB}{\mathscr{B}}
\newcommand{\cc}{{\bm c}}
\newcommand{\CC}{\field{C}}
\newcommand{\Circled}[1]
	{\tikz[baseline=-3]\node[circle,draw,scale=.7]
	{\scriptsize$#1$};}
\newcommand{\ddim}{\operatorname{\underline{dim}}}
\newcommand{\EE}{\mathbf{E}}
\newcommand{\field}[1]{\mathbf{#1}}
\newcommand{\GG}{\mathcal{G}}
\newcommand{\ii}{{\mathbf i}}

\newcommand{\IN}{\operatorname{in}}
\newcommand{\LL}{\mathcal{L}}
\newcommand{\nc}{\operatorname{nc}}
\newcommand{\nz}{\operatorname{nz}}
\newcommand{\out}{\operatorname{out}}
\newcommand{\pwp}{\mathfrak{P}}

\newcommand{\QQ}{\field{Q}}
\newcommand{\RR}{\field{R}}
\newcommand{\seg}{\operatorname{seg}}
\newcommand{\VV}{\mathbf{V}}
\newcommand{\wt}{{\rm wt}}
\newcommand{\zz}{{\bm z}}
\newcommand{\ZZ}{\field{Z}}

\theoremstyle{plain}
\newtheorem{thm}{Theorem}[section]
\newtheorem{lemma}[thm]{Lemma}
\newtheorem{prop}[thm]{Proposition}
\newtheorem{cor}[thm]{Corollary}
\theoremstyle{definition}
\newtheorem{dfn}[thm]{Definition}
\newtheorem{ex}[thm]{Example}
\newtheorem{remark}[thm]{Remark}
\numberwithin{equation}{section}
\numberwithin{figure}{section}
\numberwithin{table}{section}

\begin{document}
\title[The Gindikin-Karpelevich formula in type $A$]{A combinatorial description of the Gindikin-Karpelevich formula in type $A$}
\author{Kyu-Hwan Lee}
\address{Department of Mathematics \\ University of Connecticut \\  Storrs, CT 06269-3009.}
\email{khlee@math.uconn.edu}
\urladdr{http://www.math.uconn.edu/~khlee}
\author{Ben Salisbury}
\address{Department of Mathematics \\ University of Connecticut \\ Storrs, CT 06269-3009.}
\email{benjamin.salisbury@uconn.edu}
\urladdr{http://www.math.uconn.edu/~salisbury}
%\subjclass{}
\date{\today}
\maketitle

\begin{abstract}
A combinatorial description of the crystal $\mathcal{B}(\infty)$ for finite-dimen\-sional simple Lie algebras in terms of Young tableaux was developed by J.\ Hong and H.\ Lee. Using this description, we obtain a combinatorial rule for expressing the Gindikin-Karpelevich formula as a sum over $\mathcal{B}(\infty)$ when the underlying Lie algebra is of type $A$. We also interpret our description in terms of MV polytopes and irreducible components of quiver varieties.
\end{abstract}

\setcounter{section}{-1}
%%%%%%%%%%%%%%%%%%%%%%%
%%%%%%%%%%%%%%%%%%%%%%%
%%%%%%%%%%%%%%%%%%%%%%%
%%%%%%%%%%%%%%%%%%%%%%%
\section{Introduction}

Let $F$ be a $p$-adic field and let $N^-$ be the maximal unipotent subgroup of $\operatorname{GL}_{r+1}(F)$ with maximal torus $T$. Let $f^\circ$ denote the standard spherical vector corresponding to an unramified character $\chi$ of $T$. Let $T(\CC)$ be the maximal torus in the $L$-group $\operatorname{GL}_{r+1}(\CC)$ of $\operatorname{GL}_{r+1}(F)$, and let $\zz \in T(\CC)$ be the element corresponding to $\chi$ via the Satake isomorphism.

The Gindikin-Karpelevich formula for the longest element of the Weyl group calculates the integral of the function $f^\circ$ over $N^-(F)$ as a product over the set $\Phi^+$ of positive roots:
\begin{equation}\label{eq:GK}
\int_{N^-(F)} f^\circ({\bm n}) \,\mathrm{d}{\bm n} = \prod_{\alpha \in \Phi^+}
\frac{1-t^{-1}\zz^\alpha}{1-\zz^\alpha},
\end{equation}  where $t$ is the cardinality of the residue field of $F$.
Recently, in the works \cite{BBF:11,BN:10} of Brubaker-Bump-Friedberg and Bump-Nakasuji, the product is written as a sum over the crystal $\BB(\infty)$. (See
also \cite{McN:11}.) More precisely, they prove
\[
\prod_{\alpha\in \Phi^+} \frac{1-t^{-1}\zz^\alpha}{1-\zz^\alpha} = \sum_{b\in
\BB(\infty)} G_\ii^{(e)}(b)t^{\langle \wt(b),\rho\rangle} \zz^{-\wt(b)},
\]  
where $\rho$ is the half-sum of the positive roots, $\wt(b)$ is the weight of $b$, and the coefficients $G_\ii^{(e)}(b)$ are defined using so-called BZL paths. An important observation here is that the coefficient of $\zz^{-\wt(b)}$ is some power of $1-t^{-1}$.

This definition of the coefficients makes it necessary to compute the whole crystal graph. However, one can also define the coefficients without the need for BZL paths. In the paper \cite{KL:11}, Kim and K.-H.\ Lee adopt Lusztig's parametrization of elements of canonical basis $\CB$ and prove that
\[
\prod_{\alpha\in\Phi^+} \frac{1-t^{-1}\zz^\alpha}{1-\zz^\alpha} = \sum_{b\in \CB}
(1-t^{-1})^{\nz(\phi_\ii(b))}\zz^{-\wt(b)}.
\]
(See Proposition \ref{prop:kl} for the definition of $\nz(\phi_\ii(b))$.)  
 
In this paper, we are interested in replacing the set $\BB(\infty)$ or $\CB$ with different realizations of crystals to obtain more concrete descriptions of the coefficients in the sum. Much work has been done on realizations of crystals (e.g., \cite{Kam:10,Kang:03,KN:94,KS:97,Lit:95}). In the case of $\BB(\infty)$ for finite-dimensional simple Lie algebras, Hong and H. Lee used semistandard Young tableaux to obtain a realization of crystals \cite{HL:08}. In the first part of this paper, we will use the semistandard Young tableaux realization of type $A$ to rewrite the sum as a sum over a set $\mathcal{T}(\infty)$ of tableaux. We observe that the appropriate data to define the coefficient comes from a consecutive string of letters $k$ in the tableaux, which we call a {\em $k$-segment}. Our result is
\begin{equation} \label{eqn-first}
\prod_{\alpha\in\Phi^+} \frac{1-t^{-1}\zz^\alpha}{1-\zz^\alpha} = \sum_{b\in
\mathcal{T}(\infty) } (1-t^{-1})^{\seg(b)}\zz^{-\wt(b)},
\end{equation}
where $\seg(b)$ is the total number of $k$-segments in the tableau $b$ as $k$ varies. The main point is that the exponent $\seg(b)$ can be read off immediately from the tableau $b$.

In the second part of the paper, we use Kamnitzer's MV polytopes (\cite{Kam:07,Kam:10}) and Kashiwara and Saito's geometric construction (\cite{KS:97}) of crystals to express the sum as sums over these objects, respectively. The exponent $\seg(b)$ will have a concrete meaning in each of these realizations. Relationships among these realizations of crystals are more or less known. Therefore, new descriptions will follow from (\ref{eqn-first}) once we make necessary interpretations.

We hope to extend our results to other finite types in future work \cite{LS2:11}.

The outline of this paper is as follows. In Section \ref{sec:bases}, we briefly review the notions of Kashiwara's crystals and Lusztig's canonical bases to fix notations, and we also review BZL paths (or string parametrizations) and Lusztig's parametrizations of elements of the canonical basis. In Section \ref{sec:Binfinity-tableaux} we recall the Young tableaux realization of $\BB(\infty)$. Our main result is presented in Section \ref{sec:typeA}. In the last section, we investigate connections of the main result to MV polytopes and geometric construction of crystals.

\subsection*{Acknowledgements} K.-H.\ L.\ and B.\ S.\ thank the anonymous referees for their helpful comments and suggestions.
At the beginning of this work, K.-H.\ L.\ benefited greatly from the Banff workshop on ``Whittaker Functions, Crystal Bases, and Quantum Groups'' in June 2010 and would like to thank the organizers---B.\ Brubaker, D.\ Bump, G.\ Chinta and P.\ Gunnells. 
B.\ S.\ would like to thank Daniel Bump for sending Sage \cite{combinat,sage} code for the $\BB(\infty)$ crystal. 
\vskip 0.8 cm

%%%%%%%%%%%%%%%%%%%%%%%
%%%%%%%%%%%%%%%%%%%%%%%
%%%%%%%%%%%%%%%%%%%%%%%
%%%%%%%%%%%%%%%%%%%%%%%
\section{Canonical bases and crystals}\label{sec:bases}

Let $r\ge 1$ and suppose $\mathfrak{g} = \mathfrak{sl}_{r+1}$ with simple roots $\{ \alpha_1,\dots,\alpha_r\}$, and let $I = \{1,\dots,r\}$. Let $P$ and $P^+$ denote the wight lattice and the set of dominant integral weights, respectively. Denote by $\Phi$ and $\Phi^+$, respectively, the set of roots and the set of positive roots. Let $\{\alpha_1^\vee,\dots,\alpha_r^\vee\}$ be the set of coroots and define a pairing $\langle \ ,\ \rangle \colon P^\vee\times P \longrightarrow \ZZ$ by $\langle h,\lambda \rangle = \lambda(h)$, where $P^\vee$ is the dual weight lattice. Let $\mathfrak{h} = \CC \otimes_\ZZ P^\vee$ be the Cartan subalgebra, and let $\mathfrak{h}_\RR = \RR \otimes_\ZZ P^\vee$ be its real form.

Let $W$ be the Weyl group of $\Phi$ with simple reflections $\{\sigma_1,\dots,\sigma_r\}$. To each reduced expression $w= \sigma_{i_1}\dots \sigma_{i_k}$ for $w\in W$, we associate a {\it reduced word}, which is defined to be the $k$-tuple of positive integers $\ii = (i_1,\dots,i_k)$, and denote the set of all reduced words $\ii$ of $w\in W$ by $R(w)$. In particular, we let $w_\circ$ be the longest element of the Weyl group and call $\ii = (i_1,\dots,i_N) \in R(w_\circ)$ a {\it long word}, where $N$ is the number of positive roots.

Suppose that $q$ is an indeterminate, and let $U_q(\mathfrak{g})$ be the quantized universal enveloping algebra of $\mathfrak{g}$, which is a $\QQ(q)$-algebra generated by $e_i$, $f_i$, and $q^h$, for $i\in I$ and $h\in P^\vee$, subject to certain relations. We denote by $U_q^-(\mathfrak{g})$ the subalgebra generated by the $f_i$'s.

We write 
\[
f_i^{(c)} := \frac{f_i^c}{[c]!}, \ \ \ [c]! := \prod_{j=1}^c
\frac{q^{j}-q^{-j}}{q-q^{-1}}, \ \ \ c\in \ZZ_{> 0}.
\]
Given $\ii = (i_1,\dots,i_N)\in R(w_\circ)$ and $\cc = (c_1,\dots,c_N) \in
\ZZ_{\ge0}^N$, define
\begin{equation}\label{eq:lustparam}
f_\ii^{\cc} = f_{i_1}^{(c_1)} T_{i_1}(f_{i_2}^{(c_2)}) T_{i_1}
T_{i_2}(f_{i_3}^{(c_3)})\cdots T_{i_1} T_{i_2}\cdots T_{i_{N-1}}(f_{i_N}^{(c_N)}),
\end{equation}
where $T_i$ is the Lusztig automorphism of $U_q(\mathfrak{g})$ defined in Section 37.1.3 of \cite{Luszt:93} (there, it is denoted $T''_{i,-1}$). Then the set $\CB_\ii = \{ f_\ii ^{\cc} : \cc \in \ZZ_{\ge0}^N \}$ forms a $\QQ(q)$-basis of $U_q^-(\mathfrak{g})$, called the {\it PBW basis}. Let $\overline{\phantom{a}}\colon U_q(\mathfrak{g}) \longrightarrow U_q(\mathfrak{g})$ be the $\QQ$-algebra automorphism such that
\[
e_i \mapsto e_i, \ \ \ \ f_i \mapsto f_i, \ \ \ \ q\mapsto q^{-1}, \ \ \ \ q^h \mapsto
q^{-h},
\]
for $i \in I$ and $h\in P^\vee$.

\begin{prop}[\cite{Luszt:90}] \label{prop:Lusz}
Assume that $\ii \in R(w_\circ)$.
\begin{enumerate}
\item Assume that $\mathscr{L}_\ii$ is the $\ZZ[q]$-span of the basis $\CB_\ii$. Then $\mathscr{L}_\ii$ is independent of $\ii \in R(w_\circ)$, so we denote it simply by $\mathscr{L}$.
\item Suppose that $\pi\colon \mathscr{L} \longrightarrow \mathscr{L}/q\mathscr{L}$ is the canonical projection. Then $\pi(\CB_\ii)$ is a $\ZZ$-basis of $\mathscr{L}/q\mathscr{L}$, independent of $\ii$. Moreover, the restriction of $\pi$ to $\mathscr{L}\cap\overline{\mathscr{L}}$ is an isomorphism of $\ZZ$-modules $\overline\pi \colon \mathscr{L}\cap\overline{\mathscr{L}} \longrightarrow \mathscr{L}/q\mathscr{L}$, and $\CB = \overline\pi^{-1}(\pi(\CB_\ii))$ is a $\QQ(q)$-basis of $U_q^-(\mathfrak{g})$.
\end{enumerate}
\end{prop}

The basis $\CB$ is called the {\it canonical basis} of $U_q^-(\mathfrak{g})$. Note that each $b \in \CB$ satisfies $b \equiv f_\ii^\cc \bmod q\mathscr{L}$ for some $\ii\in R(w_\circ)$ and $\cc\in\ZZ_{\ge0}^N$.

\medskip

Let $\widetilde e_{i}$, $\widetilde f_{i}$ be the Kashiwara operators on $U_q^-(\mathfrak{g})$ defined in \cite{Kash:91}. Let $\mathcal{A} \subset \QQ(q)$ be the subring of functions regular at $q=0$ and define $\LL(\infty)$ to be the $\mathcal{A}$-lattice spanned by
\[
S= \{ \widetilde f_{i_1} \widetilde f_{i_2} \cdots \widetilde f_{i_t} \cdot 1 \in U_q^-(\mathfrak{g}) : t\ge 0, \ i_k \in I \}.
\]

\begin{prop}[\cite{Kash:91}] \hfill
\begin{enumerate}
\item Let $\pi' \colon \LL(\infty) \longrightarrow \LL(\infty)/q\LL(\infty)$ be the natural projection and set $\BB(\infty) = \pi'(S)$. Then $\BB(\infty)$ is a $\QQ$-basis of $\LL(\infty)/q\LL(\infty)$.
\item The operators $\widetilde e_i$ and $\widetilde f_i$ act on $\LL(\infty)/q\LL(\infty)$ for each $i\in I$. Moreover, $\widetilde e_i \colon \BB(\infty) \longrightarrow \BB(\infty)\sqcup \{0\}$ and $\widetilde f_i\colon \BB(\infty) \longrightarrow \BB(\infty)$ for each $i\in I$. For $b,b'\in \BB(\infty)$, we have $\widetilde f_i b = b'$ if and only if $\widetilde e_ib' = b$. 
\item For each $b \in \BB(\infty)$, there is a unique element $\GG(b) \in \LL(\infty)\cap\overline{\LL(\infty)}$ such that $\pi'(\GG(b)) = b$. The set $\GG(\infty) = \{ \GG(b) : b \in \BB(\infty) \}$ forms a basis of $U_q^-(\mathfrak{g})$.
\end{enumerate}
\end{prop}

The basis $\BB(\infty)$ is called the {\it crystal basis} of $U_q^-(\mathfrak{g})$, and the basis $\GG(\infty)$ is called the {\it global crystal basis} of $U_q^-(\mathfrak{g})$. In \cite{GL:93}, Grojnowski and Lusztig showed that $\GG(\infty) = \CB$. However, there are two different, but standard, ways to parametrize elements of a canonical basis or a global crystal basis. For a choice of $\ii \in R(w_\circ)$, there is a unique path, called {\it BZL path}, from a crystal element $b$ to the unique weight zero crystal element $b_\infty$. The parametrization coming from BZL paths is called the {\it string parametrization} of $b$, which we will denote by $\psi_\ii(b)$. See the definition below. On the other hand, each canonical basis element comes from some $f_\ii^\cc \in \CB_\ii$ as in Proposition \ref{prop:Lusz}. From this we obtain a parametrization $\cc\in\ZZ_{\ge0}^N$ of the element in the canonical basis. This latter parametrization is called the {\it Lusztig parametrization} of $b\in \CB$, and we denote it by $\phi_\ii(b)$. Berenstein and Zelevinsky calculated a way to link these parametrizations \cite{BZ:96}. The connection between these two parametrizations is crucial to our arguments below.

\medskip

One may define the notion of a crystal abstractly. A {\it $U_q(\mathfrak{g})$-crystal} is a set $\BB$ together with maps
\[
\wt\colon \BB \longrightarrow P, \ \ \ \ \ \
\widetilde e_i, \widetilde f_i\colon \BB \longrightarrow \BB\sqcup\{0\},\ \ \ \ \ \
\varepsilon_i,\varphi_i\colon \BB \longrightarrow \ZZ\sqcup\{-\infty\},
\] 
that satisfy a certain set of axioms (see, e.g.\ \cite{HK:02}), and a {\it crystal morphism} is defined in a natural way. We recall the tensor product of crystals and the signature rule, which are necessary to understand the combinatorics of $\BB(\infty)$.

\begin{dfn}
Let $\BB_1$ and $\BB_2$ be $U_q(\mathfrak{g})$-crystals. Then the {\it tensor product} of crystals $\BB_1\otimes \BB_2$ is $\BB_1 \times \BB_2$ as a set, endowed with the following crystal structure. The Kashiwara operators are given by
\begin{align*}
\widetilde e_i(b_1\otimes b_2)
&= \begin{cases}
\widetilde e_ib_1 \otimes b_2 & \text{ if } \varphi_i(b_1) \ge \varepsilon_i(b_2),\\
b_1 \otimes \widetilde e_ib_2 & \text{ otherwise},
\end{cases}\\
\widetilde f_i(b_1\otimes b_2)
&= \begin{cases}
\widetilde f_ib_1 \otimes b_2 & \text{ if } \varphi_i(b_1) > \varepsilon_i(b_2),\\
b_1 \otimes \widetilde f_ib_2 & \text{ otherwise}.
\end{cases}
\end{align*}
We also have
\begin{align*}
\wt(b_1\otimes b_2) &= \wt(b_1)+\wt(b_2),\\
\varepsilon_i(b_1\otimes b_2) &= \max\bigl(\varepsilon_i(b_1),\varepsilon_i(b_2) - \langle \alpha_i^\vee, \wt(b_1)\rangle\bigr),\\
\varphi_i(b_1\otimes b_2) &= \max\bigl(\varphi_i(b_2),\varphi_i(b_1)+\langle \alpha_i^\vee,\wt(b_2)\rangle\bigr).
\end{align*}
\end{dfn}

Using the tensor product rule above, one obtains a way to determine the component of a tensor product on which a Kashiwara operator acts, called the {\it signature rule}. Let $i\in I$ and set $\BB = \BB_1 \otimes \cdots \otimes \BB_m$. Take $b = b_1\otimes \cdots \otimes b_m \in \BB$. To calculate either $\widetilde e_i$ or $\widetilde f_i$, create a sequence of $+$ and $-$ according to
\[
(\ \underbrace{-\cdots-}_{\varepsilon_i(b_1)},\underbrace{+\cdots+}_{\varphi_i(b_1)},
\cdots ,
\underbrace{-\cdots-}_{\varepsilon_i(b_m)},\underbrace{+\cdots+}_{\varphi_i(b_m)} \ )
\]
Cancel any $+-$ pair to obtain a sequence of $-$'s followed by $+$'s. We call the resulting sequence the {\it $i$-signature} of $b$, and denote it by $i$-sgn$(b)$. Then $\widetilde e_i$ acts on the component of $b$ corresponding to the rightmost $-$ in $i$-sgn$(b)$ and $\widetilde f_i$ acts on the component of $b$ corresponding to the leftmost $+$ in $i$-sgn$(b)$. If there is no remaining $-$ (or $+$,  respectively) in $i$-sgn$(b)$ then we have $\widetilde e_i(b)=0$ (or $\widetilde f_i(b) =0$, respectively). 

As an illustration, we apply this rule to the semistandard Young tableaux realization of $U_q(\mathfrak{sl}_{r+1})$-crystals $\BB(\lambda)$ of highest weight representations for $\lambda$ a dominant integral weight. This description is according to Kashiwara and Naka\-shima. See \cite{KN:94} or \cite{HK:02} for the details of this construction including precise definitions of $\varepsilon_i, \varphi_i, \mathrm{wt}$ in this case.
For the fundamental weight $\Lambda_1$, the crystal graph of $\BB(\Lambda_1)$ is given by
\[
\BB(\Lambda_1): \ \ \ \ 
\begin{tikzpicture}[scale=1.5,baseline=-4]
 \node (1) at (0,0) {$\boxed{1}$};
 \node (2) at (1,0) {$\boxed{2}$};
 \node (d) at (2,0) {$\cdots$};
 \node (n-1) at (3,0) {$\boxed{r}$};
 \node (n) at (4.2,0) {$\boxed{r+1}$};
 \draw[->] (1) to node[above]{\tiny$1$} (2);
 \draw[->] (2) to node[above]{\tiny$2$} (d);
 \draw[->] (d) to node[above]{\tiny$r-1$} (n-1);
 \draw[->] (n-1) to node[above]{\tiny$r$} (n);
\end{tikzpicture}
\]
Using this fundamental crystal $\BB(\Lambda_1)$, we may understand any tableaux of shape $\lambda$ by embedding the corresponding crystal $\BB(\lambda)$ into
$\BB(\Lambda_1)^{\otimes m}$, where $m$ is the number of boxes in the $\lambda$ shape. For example, in type $A_4$, we have
\[
\BB(\Lambda_1+\Lambda_2+\Lambda_3) \ni b= \young(133,34,5) \mapsto
\young(3)\otimes\young(3)\otimes\young(4)\otimes\young(1)\otimes\young(3)\otimes\young(5)
\in \BB(\Lambda_1)^{\otimes 6}.
\]
With this image of the embedding, we may apply the signature rule to determine on which box $\widetilde f_i$ and $\widetilde e_i$ act. In this case, with $i=3$, we have $3$-$\operatorname{sgn}(b) = (+,+,-,\cdot,+,\cdot) = (+,\cdot,\cdot,\cdot,+,\cdot)$. Thus $\widetilde e_3b = 0$ and
\[
\widetilde f_3b = \young(134,34,5).
\]

For a given $\ii =(i_1, i_2, \dots, i_N) \in R(w_\circ)$, define the {\it BZL path} of $b\in \BB(\infty)$ as follows. Define $a_1$ to be the maximal integer such that $\widetilde e_{i_1}^{a_1} b \neq 0$. Then let $a_2$ be the maximal integer such that $\widetilde e_{i_2}^{a_2} \widetilde e_{i_1}^{a_1} b \neq 0$. Inductively, let $a_j$ be the maximal integer such that
\[
\widetilde e_{i_j}^{a_j} \widetilde e_{i_{j-1}}^{a_{j-1}} \cdots \widetilde e_{i_2}^{a_2} \widetilde e_{i_1}^{a_1} b \neq 0,
\]
for $j=1,\dots,N$.  Then we define $\psi_\ii(b) = (a_1,\dots,a_N)$.  Let
\[
C_\ii = \{ \psi_\ii(b) : b\in \BB(\infty) \}.
\]
The BZL paths are also known as {\it string parametrizations} or {\it Kashiwara data} in the literature (see, for example \cite{BZ:01,Kam:07}). The associated cones were studied by Littelmann in \cite{Lit:98}.  In particular, for $\ii=(i_1,i_2,\dots,i_N)\in R(w_\circ)$ and $b\in \BB(\infty)$, it is known that $\widetilde e_{i_N}^{a_N} \cdots \widetilde e_{i_2}^{a_2} \widetilde e_{i_1}^{a_1}b = b_\infty$, where $b_\infty$ is the unique element of weight zero in $\BB(\infty)$.

\medskip

In order to prove that one can obtain the coefficients in the expansion of the product in the Gindikin-Karpelevich formula using crystals of Young tableaux, we will need to first write the Gindikin-Karpelevich formula as a sum over elements of Lusztig's canonical basis, as shown in \cite{KL:11}.

\begin{prop}[{\cite{KL:11}}]\label{prop:kl}
Let $\CB$ be Lusztig's canonical basis and let $\ii \in R(w_\circ)$.  Then 
\[
\prod_{\alpha \in \Phi^+} \frac{1-t^{-1}\zz^{\alpha}}{1-\zz^\alpha} = \sum_{b\in\CB}
(1-t^{-1})^{\nz(\phi_\ii(b))}\zz^{-\wt(b)},
\]
where $\phi_\ii\colon \CB \longrightarrow \ZZ_{\ge0}^N$ is the map which takes elements in the canonical basis to their Lusztig parametrization and $\nz(\phi_\ii(b))$ is the number of nonzero elements in the sequence $\phi_\ii(b)$.
\end{prop}

\begin{remark}
Proposition \ref{prop:kl} holds for any finite-dimensional simple Lie algebra.
\end{remark}

Now we need a way to change BZL paths of elements in $\BB(\infty)$ to Lusztig parametrizations of elements in $\CB$. The word we consider is
\[
\ii_{A_r} := (1,2,1,3,2,1,\dots,r,r-1,\dots,2,1) \in R(w_\circ).
\]
From here to the end of Section \ref{sec:typeA}, any dependence on $\ii$ will assume that $\ii = \ii_{A_r}$.

Associated to each entry in a given BZL path of a highest weight crystal $\BB(\lambda)$ is a decoration: a circle, a box, both a circle and a box, or neither. However, boxing does not occur in $\BB(\infty)$ (see \cite{BN:10}), so we only describe the circling rule. In type $A$, we write the BZL paths in triangles of the following form:
\begin{equation}\label{eq:tri1}
\begin{array}{c@{}c@{}c@{}c@{}c@{}c@{}c}
&&&a_1&&&\\
&&a_2&&a_3&&\\
&a_4&&a_5&&a_6&\\
\iddots&&\vdots&&\vdots&&\ddots
\end{array} 
\end{equation}
It will be beneficial to write the triangular arrays using matrix indices, so reindex the above in the following way:
\begin{equation}\label{eq:triangles}
\begin{array}{c@{}c@{}c@{}c@{}c@{}c@{}c}
&&&a_{1,1}&&&\\
&&a_{2,1}&&a_{2,2}&&\\
&a_{3,1}&&a_{3,2}&&a_{3,3}&\\
\iddots&&\vdots&&\vdots&&\ddots
\end{array} 
\end{equation}
This triangular array look more natural if we recall \cite{Lit:98} that \[a_{1,1} \ge 0; \quad a_{2,1} \ge a_{2,2} \ge 0 ; \quad a_{3,1} \ge a_{3,2} \ge a_{3,3} \ge 0 ; \quad  \ldots .\]  
If the entry $a_{j,\ell-1} = a_{j,\ell}$, then we circle $a_{j,\ell-1}$. We understand that the entries outside the triangle are zero, so the rightmost entry of a row is circled if it is zero. Moreover, we call the {\it $j$th row} of $\psi_\ii(b)$ the row which starts with $a_{j,1}$. Finally, to express this triangle in an inline form, we write $(a_{1,1};a_{2,1},a_{2,2};\dots;a_{r,1},\dots,a_{r,r})$.

\begin{ex}
Let $b_\infty = \widetilde e_1 \widetilde e_2^2 \widetilde e_3^2 \widetilde e_4^4 \widetilde e_2^2 \widetilde e_3^3 \widetilde e_1 \widetilde e_2 \, b$.  Then
\[
\psi_\ii(b) = 
\begin{array}{ccccccc}
&&&\Circled{0}&&&\\
&&\Circled{1}&&1&&\\
&3&&2&&\Circled{0}&\\
4&&\Circled{2}&&2&&1
\end{array}
=(\Circled{0};\Circled{1},1;3,2,\Circled{0}; 4, \Circled{2}, 2, 1).
\]
\end{ex}

The following proposition is crucial and immediately implies Corollary \ref{cor:BN} given below.

\begin{prop}[{\cite{BZ:96}}]\label{prop:Lus-to-BZL-A}
The map $\sigma_\ii\colon C_\ii \longrightarrow \ZZ_{\ge0}^N$, which takes the BZL path of an element $b\in \BB(\infty)$ to its corresponding Lusztig parametrization, is given by
\[
(a_{1,1},\dots,a_{r,r}) \mapsto (a_{1,1};a_{2,2},a_{2,1}-a_{2,2};\dots;a_{r,r},
a_{r,r-1}-a_{r,r}, \dots, a_{r,1} - a_{r,2}).
\]
\end{prop}

\begin{cor} [\cite{BN:10, KL:11}] \label{cor:BN}  Let $\ii \in R(w_\circ)$.
\begin{enumerate} 
\item The number of circled entries in a BZL path is the same as the number of zero entries in the corresponding Lusztig parametrization.
\item  We have 
\[
\prod_{\alpha \in \Phi^+} \frac{1-t^{-1}\zz^{\alpha}}{1-\zz^\alpha} =
\sum_{b\in\BB(\infty)} (1-t^{-1})^{\nc(\psi_\ii(b))}\zz^{-\wt(b)},
\] 
where $\nc(\psi_\ii(b))$ is the number of uncircled entries in $\psi_\ii(b)$.
\end{enumerate}
\end{cor}

\vskip 0.8 cm

%%%%%%%%%%%%%%%%%%%%%%%
%%%%%%%%%%%%%%%%%%%%%%%
%%%%%%%%%%%%%%%%%%%%%%%
%%%%%%%%%%%%%%%%%%%%%%%
\section{A combinatorial realization of $\BB(\infty)$}\label{sec:Binfinity-tableaux}

This section is a summary of the results for type $A$ from \cite{HL:08}. Recall that a tableaux $b$ is semistandard if entries are weakly increasing in rows and strictly increasing in columns. Hong and H. Lee define a tableau $b$ to be {\it marginally large} if, for all $1 \le i \le r$, the number of $i$-boxes in the $i$th row of $b$ is greater than the number of all boxes in the $(i+1)$st row by exactly one.

We define $\mathcal{T}(\infty)$ to be the set of tableaux $b$ satisfying the following conditions.
\begin{enumerate}
\item Entries in $b$ come from the alphabet $\{ 1 \prec 2 \prec \cdots \prec r+1\}$.
\item $b$ is semistandard and consists of $r$ rows.
\item For $1 \le j \le r$, the $j$th row of the leftmost column of $b$ is a $j$-box.
\item $b$ is marginally large.
\end{enumerate}

To obtain the crystal structure of $\mathcal{T}(\infty)$, it remains to describe how the Kashiwara operators act on tableaux in $\mathcal{T}(\infty)$. The main difference between this procedure and the procedure to compute the Kashiwara operators in a finite crystal is that we require each vertex to be a marginally large tableaux, so the shape of the tableaux varies as one moves down the crystal. Indeed, to calculate $\widetilde f_ib$, $i\in I$, for some $b\in \BB(\infty)$, we apply the following procedure.
\begin{enumerate}
\item Apply $\widetilde f_i$ to $b$ using the $i$-signature of $b$ as usual.
\item If the result is marginally large, then we are done. If not, it is the case that $\widetilde f_i$ is applied to the rightmost $i$-box in the $i$th row. Insert one column consisting of $i$ rows to the left of the box $\widetilde f_i$ acted on. This new column should have a $k$-box in the $k$th row, for $1 \le k \le i$.
\end{enumerate}
Similarly, to calculate $\widetilde e_ib$, one does the following.
\begin{enumerate}
\item Apply $\widetilde e_i$ to $b$ using the $i$-signature of $b$ as usual.
\item If the result is marginally large or zero, then we are done. If not, it is the case that $\widetilde e_i$ is applied to the box to the right of the rightmost $i$-box in the $i$th row. Remove the column containing the changed box, which is a column of $i$ rows having a $k$-box in the $k$th row, for $1 \le k \le i$.
\end{enumerate}

\begin{prop}[\cite{HL:08}]
We have $\mathcal{T}(\infty) \cong \BB(\infty)$ as crystals.
\end{prop}

\begin{ex}
For $r=3$, the elements of $\mathcal{T}(\infty)$ all have the form
\[
b = \begin{tikzpicture}[baseline]
\matrix [matrix of math nodes,column sep=-.4, row sep=-.4] 
 {
	\node[draw,fill=gray!25]{1}; & 
	\node[draw,fill=gray!25]{1 \cdots 1}; & 
	\node[draw,fill=gray!25]{1}; & 
	\node[draw,fill=gray!25]{1\cdots 1}; & 
	\node[draw,fill=gray!25]{1\cdots1}; & 
	\node[draw,fill=gray!25]{1}; & 
	\node[draw]{2\cdots 2}; & 
	\node[draw]{3\cdots 3}; & 
	\node[draw]{4\cdots 4};\\
	\node[draw,fill=gray!25]{2}; & 
	\node[draw,fill=gray!25]{2\cdots 2}; & 
	\node[draw,fill=gray!25]{2}; &
	\node[draw]{3 \cdots 3}; & 
	\node[draw]{4\cdots 4}; \\
  	\node[draw,fill=gray!25]{3}; & 
  	\node[draw]{4\cdots 4}; \\
 };
\end{tikzpicture},
\]
where the shaded parts are the required parts and the unshaded parts are variable. In particular, the unique element of weight zero in this crystal is
\[
b_\infty = \begin{tikzpicture}[baseline]
\matrix [matrix of math nodes,column sep=-.4, row sep=-.4] 
 {
	\node[draw,fill=gray!25]{1}; & 
	\node[draw,fill=gray!25]{1}; &
	\node[draw,fill=gray!25]{1}; \\
  	\node[draw,fill=gray!25]{2}; & 
  	\node[draw,fill=gray!25]{2};  \\
  	\node[draw,fill=gray!25]{3};  \\
 };
\end{tikzpicture}.
\]
\end{ex}

Following Bump and Nakasuji in \cite{BN:10}, we wish to suppress the required columns from the tableaux and only include the variable parts. This convention will save space, making drawing the graphs easier and it will help make the $k$-segments, to be defined later, stand out. We will call this modification of $b\in \mathcal{T}(\infty)$ the {\it reduced form} of $b$, and denote it by $b^\sharp$. For example, with $r=3$, we have
\[
\left(\,
\begin{tikzpicture}[baseline]
\matrix [matrix of math nodes,column sep=-.4, row sep=-.4,text height=7pt] 
 {
	\node[draw,fill=gray!25]{1}; & 
	\node[draw,fill=gray!25]{1}; &
	\node[draw,fill=gray!25]{1}; \\
  	\node[draw,fill=gray!25]{2}; & 
  	\node[draw,fill=gray!25]{2}; \\
  	\node[draw,fill=gray!25]{3}; \\
 };
\end{tikzpicture}
\,\right)^\sharp = 
\begin{tikzpicture}[baseline]
\matrix [matrix of math nodes,column sep=-.4, row sep=-.4,text height=7pt] 
 {
 	\node[draw]{*};  \\
  	\node[draw]{*};  \\
  	\node[draw]{*};  \\
 };
\end{tikzpicture}, \] 
\[
\left(\,
\begin{tikzpicture}[baseline]
\matrix [matrix of math nodes,column sep=-.4, row sep=-.4,text height=7pt] 
 {
	\node[draw,fill=gray!25]{1}; & 
	\node[draw,fill=gray!25]{1}; &
	\node[draw,fill=gray!25]{1}; &
	\node[draw,fill=gray!25]{1}; &
	\node[draw,fill=gray!25]{1}; &	
	\node[draw]{2};  \\
  	\node[draw,fill=gray!25]{2}; & 
	\node[draw,fill=gray!25]{2}; &
	\node[draw,fill=gray!25]{2}; &
	\node[draw,fill=gray!25]{2}; \\ 
	\node[draw,fill=gray!25]{3}; &
	\node[draw]{4}; &
  	\node[draw]{4}; \\
 };
\end{tikzpicture}
\,\right)^\sharp = 
\begin{tikzpicture}[baseline]
\matrix [matrix of math nodes,column sep=-.4, row sep=-.4,text height=7pt] 
 {
 	\node[draw]{2};  \\
  	\node[draw]{*};  \\
  	\node[draw]{4}; & \node[draw]{4}; \\
 };
\end{tikzpicture},
\]
\[
\left(\,
\begin{tikzpicture}[baseline]
\matrix [matrix of math nodes,column sep=-.4, row sep=-.4,text height=7pt] 
 {
	\node[draw,fill=gray!25]{1}; &
	\node[draw,fill=gray!25]{1}; &
	\node[draw,fill=gray!25]{1}; &
	\node[draw,fill=gray!25]{1}; &
	\node[draw,fill=gray!25]{1}; &
	\node[draw,fill=gray!25]{1}; &	
	\node[draw]{3}; \\
  	\node[draw,fill=gray!25]{2}; & 
	\node[draw,fill=gray!25]{2}; &
	\node[draw,fill=gray!25]{2}; &
	\node[draw]{3}; &	
	\node[draw]{3}; \\ 
	\node[draw,fill=gray!25]{3}; &
	\node[draw]{4}; \\
 };
\end{tikzpicture}
\,\right)^\sharp = 
\begin{tikzpicture}[baseline]
\matrix [matrix of math nodes,column sep=-.4, row sep=-.4,text height=7pt] 
 {
    \node[draw]{3}; \\
    \node[draw]{3}; &
    \node[draw]{3}; \\
    \node[draw]{4}; \\
 };
\end{tikzpicture},
\]
where $*$ should be considered as void. In particular, the resulting shape need not be a Young diagram. Denote by $\mathcal{T}(\infty)^\sharp$ the set of all reduced forms of $b\in \mathcal{T}(\infty)$.

\vskip 0.8 cm

%%%%%%%%%%%%%%%%
%%%%%%%%%%%%%%%%
%%%%%%%%%%%%%%%%
%%%%%%%%%%%%%%%%
\section{Main result}\label{sec:typeA}

In this section, we state and prove our main result. Our description of the coefficients in the sum will rely on certain patterns of boxes in a Young tableau.

\begin{dfn}
Let $b\in \mathcal{T}(\infty)$. Define a {\it $k$-segment}, $2\le k \le r+1$, to be part of a row from $b$ of the form
\[
\begin{tikzpicture}[baseline=-5]
\matrix [matrix of math nodes,nodes=draw,text height=8pt]
 { k & \cdots & k \\};
\end{tikzpicture}
\]
Moreover, we do not consider the required collection of $k$-boxes beginning the $k$th row of $b$ to be a $k$-segment; that is, we only consider $k$-boxes that appear in $b^\sharp$. Define $\seg_k(b)$ to be the number of $k$-segments in $b$, and let
\begin{equation}\label{eq:cb}
\seg(b) := \sum_{k=2}^{r+1} \seg_k(b).
\end{equation}
We also say a $k$-segment has {\it length} $m$ if the $k$-segment consists of $m$ boxes.
\end{dfn}
According to the definition, there are no $1$-segments, and a $k$-segment can only occur in rows $1$ through $k-1$. With this definition, we now state: 

\begin{thm}\label{thm:main-A}
Let $\Phi$ be the root system of $\mathfrak{sl}_{r+1}$ and let $\mathcal{T}(\infty)$ be the set of marginally large tableaux defined above. Then
\begin{equation}\label{eq:main-A}
\prod_{\alpha\in\Phi^+} \frac{1-t^{-1}\bm z^{\alpha}}{1-\bm z^{\alpha}} = \sum_{b\in \mathcal{T}(\infty)} (1-t^{-1})^{\seg(b)} \bm z^{-\wt(b)}.
\end{equation}
\end{thm}

\medskip

Before we present the proof of the theorem, we first give an example and two lemmas.

\begin{ex}\label{ex:r=2}
Let $r=2$. Then top part of $\mathcal{T}(\infty)$ is shown in Figure \ref{fig:b2} with corresponding coefficients shown in Figure \ref{fig:b2q}.

\begin{sidewaysfigure}[p]
\begin{tikzpicture}[yscale=2.2,xscale=2.7]
\node (000) at (0,0) {$\young(*,*)$};
\node (100) at (-1,-1) {$\young(2,*)$};
\node (010) at (1,-1) {$\young(*,3)$};
\node (200) at (-2,-2) {$\young(22,*)$};
\node (011) at (-.25,-2) {$\young(3,*)$};
\node (110) at (.25,-2) {$\young(2,3)$};
\node (020) at (2,-2) {$\young(*,33)$};
\node (300) at (-3,-3) {$\young(222,*)$};
\node (111) at (-1.9,-3) {$\young(23,*)$};
\node (210) at (-.75,-3) {$\young(22,3)$};
\node (021) at (.75,-3) {$\young(3,3)$};
\node (120) at (1.9,-3) {$\young(2,33)$};
\node (030) at (3,-3) {$\young(*,333)$};
\node (400) at (-3.2,-4) {$\young(2222,*)$};
\node (211) at (-2.2,-4) {$\young(223,*)$};
\node (022) at (-0.5,-4) {$\young(33,*)$};
\node (310) at (-1.7,-4) {$\young(222,3)$};
\node (121) at (0,-4) {$\young(23,3)$};
\node (220) at (0.5,-4) {$\young(22,33)$};
\node (031) at (1.7,-4) {$\young(3,33)$};
\node (130) at (2.2,-4) {$\young(2,333)$};
\node (040) at (3.2,-4) {$\young(*,3333)$};
\path[->,font=\scriptsize]
 (000) edge node[midway,fill=white]{$1$} (100)
 (000) edge node[midway,fill=white]{$2$} (010)
 (100) edge node[midway,fill=white]{$1$} (200)
 (100) edge node[midway,fill=white]{$2$} (011)
 (010) edge node[midway,fill=white]{$1$} (110)
 (010) edge node[midway,fill=white]{$2$} (020)
 (200) edge node[midway,fill=white]{$1$} (300)
 (200) edge node[midway,fill=white]{$2$} (111)
 (011) edge node[midway,fill=white]{$1$} (111)
 (011) edge node[near end,fill=white]{$2$} (021)
 (110) edge node[near end,fill=white]{$1$} (210)
 (110) edge node[midway,fill=white]{$2$} (120)
 (020) edge node[midway,fill=white]{$1$} (120)
 (020) edge node[midway,fill=white]{$2$} (030)
 (300) edge node[midway,fill=white]{$1$} (400)
 (300) edge node[midway,fill=white]{$2$} (211)
 (111) edge node[midway,fill=white]{$1$} (211)
 (111) edge node[near end,fill=white]{$2$} (022)
 (210) edge node[near end,fill=white]{$1$} (310)
 (210) edge node[midway,fill=white]{$2$} (121)
 (021) edge node[midway,fill=white]{$1$} (121)
 (021) edge node[near start,fill=white]{$2$} (031)
 (120) edge node[near start,fill=white]{$1$} (220)
 (120) edge node[near start,fill=white]{$2$} (130)
 (030) edge node[midway,fill=white]{$1$} (130)
 (030) edge node[midway,fill=white]{$2$} (040);
\end{tikzpicture}
\caption{The top part of $\mathcal{T}(\infty)^\sharp$ when $r=2$.}
\label{fig:b2}
\end{sidewaysfigure}
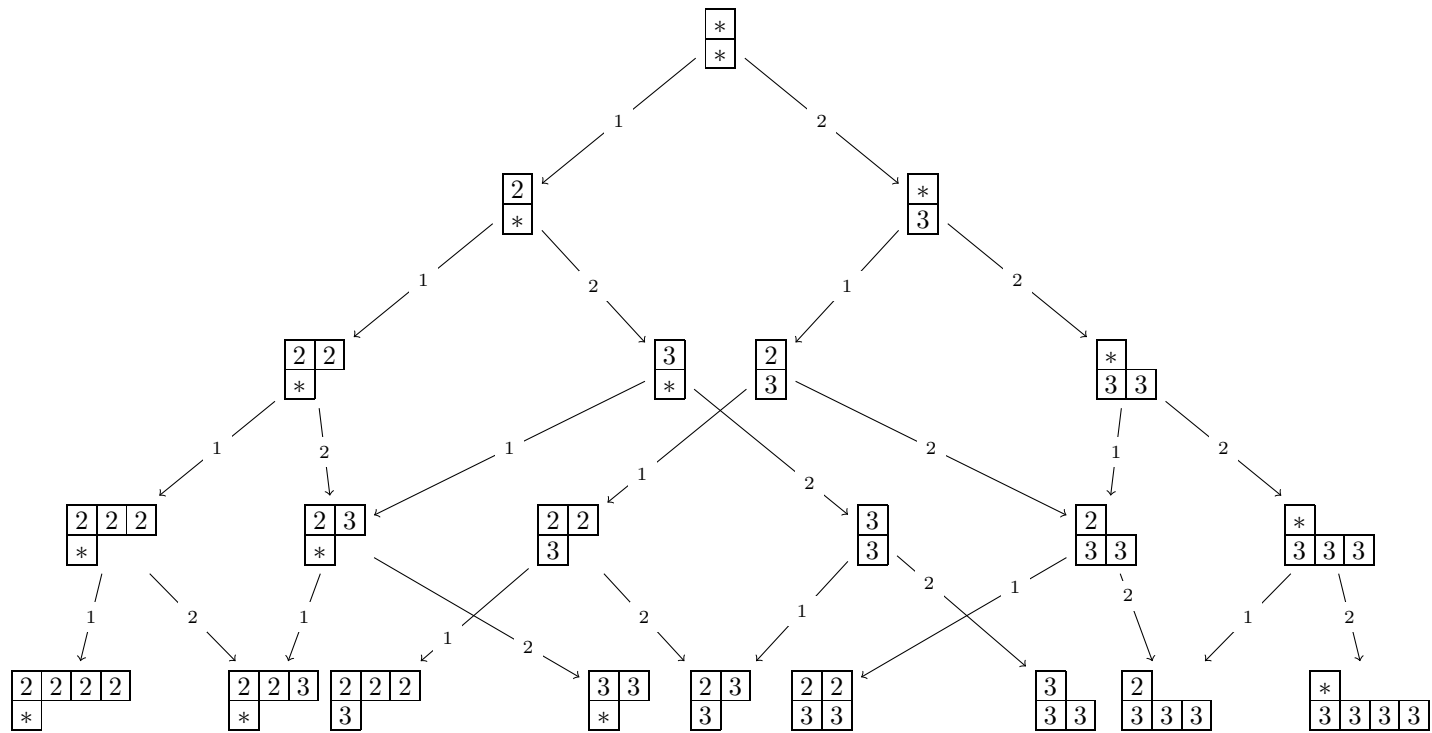

\begin{sidewaysfigure}[p]
\begin{tikzpicture}[yscale=2.2,xscale=2.7]
\node (000) at (0,0) {$1$};
\node (100) at (-1,-1) {$1-t^{-1}$};
\node (010) at (1,-1) {$1-t^{-1}$};
\node (200) at (-2,-2) {$1-t^{-1}$};
\node (011) at (-.3,-2) {$1-t^{-1}$};
\node (110) at (.3,-2) {$(1-t^{-1})^2$};
\node (020) at (2,-2) {$1-t^{-1}$};
\node (300) at (-2.7,-3) {$1-t^{-1}$};
\node (111) at (-1.9,-3) {$(1-t^{-1})^2$};
\node (210) at (-.75,-3) {$(1-t^{-1})^2$};
\node (021) at (.75,-3) {$(1-t^{-1})^2$};
\node (120) at (1.9,-3) {$(1-t^{-1})^2$};
\node (030) at (2.7,-3) {$1-t^{-1}$};
\node (400) at (-3.2,-4) {$1-t^{-1}$};
\node (211) at (-2.35,-4) {$(1-t^{-1})^2$};
\node (022) at (-0.7,-4) {$1-t^{-1}$};
\node (310) at (-1.7,-4) {$(1-t^{-1})^2$};
\node (121) at (0,-4) {$(1-t^{-1})^3$};
\node (220) at (0.7,-4) {$(1-t^{-1})^2$};
\node (031) at (1.7,-4) {$(1-t^{-1})^2$};
\node (130) at (2.35,-4) {$(1-t^{-1})^2$};
\node (040) at (3.2,-4) {$1-t^{-1}$};
\path[->,font=\scriptsize]
 (000) edge node[midway,fill=white]{$1$} (100)
 (000) edge node[midway,fill=white]{$2$} (010)
 (100) edge node[midway,fill=white]{$1$} (200)
 (100) edge node[midway,fill=white]{$2$} (011)
 (010) edge node[midway,fill=white]{$1$} (110)
 (010) edge node[midway,fill=white]{$2$} (020)
 (200) edge node[midway,fill=white]{$1$} (300)
 (200) edge node[midway,fill=white]{$2$} (111)
 (011) edge node[midway,fill=white]{$1$} (111)
 (011) edge node[near end,fill=white]{$2$} (021)
 (110) edge node[near end,fill=white]{$1$} (210)
 (110) edge node[midway,fill=white]{$2$} (120)
 (020) edge node[midway,fill=white]{$1$} (120)
 (020) edge node[midway,fill=white]{$2$} (030)
 (300) edge node[midway,fill=white]{$1$} (400)
 (300) edge node[midway,fill=white]{$2$} (211)
 (111) edge node[midway,fill=white]{$1$} (211)
 (111) edge node[near end,fill=white]{$2$} (022)
 (210) edge node[near end,fill=white]{$1$} (310)
 (210) edge node[midway,fill=white]{$2$} (121)
 (021) edge node[midway,fill=white]{$1$} (121)
 (021) edge node[near start,fill=white]{$2$} (031)
 (120) edge node[near start,fill=white]{$1$} (220)
 (120) edge node[near start,fill=white]{$2$} (130)
 (030) edge node[midway,fill=white]{$1$} (130)
 (030) edge node[midway,fill=white]{$2$} (040);
\end{tikzpicture}
\caption{The coefficients for the top part of $\mathcal{T}(\infty)$ when $r=2$.}
\label{fig:b2q}
\end{sidewaysfigure}

Let us consider the element
\[ b= \begin{tikzpicture}[baseline]
\matrix [matrix of math nodes,column sep=-.4, row sep=-.4,text height=7pt] 
 {
	\node[draw,fill=gray!25]{1}; &
	\node[draw,fill=gray!25]{1}; &
	\node[draw,fill=gray!25]{1}; & 
	\node[draw]{2}; &
	\node[draw]{3}; \\
  	\node[draw,fill=gray!25]{2}; &
  	\node[draw]{3};  \\
 };
\end{tikzpicture} \qquad \text{ with } \qquad
b^\sharp = \young(23,3)\ .
\]
There is one $2$-segment in the first row and one $3$-segment in each of the first and second rows, so $\seg_2(b) = 1$ and $\seg_3(b) = 2$. Thus $\seg(b) = 1 + 2 = 3$ and the coefficient corresponding to $b$ is $(1-t^{-1})^3$. Notice that the decorated BZL path of $b$ is $\psi_\ii(b) = (1;2,1)$ with no circle, which results in $(1-t^{-1})^3$.

Now consider
\[b= \begin{tikzpicture}[baseline]
\matrix [matrix of math nodes,column sep=-.4, row sep=-.4,text height=7pt] 
 {
	\node[draw,fill=gray!25]{1}; &
	\node[draw,fill=gray!25]{1}; &
	\node[draw]{3}; \\
  	\node[draw,fill=gray!25]{2}; \\
 };
\end{tikzpicture} \qquad \text{ with } \qquad
(b')^\sharp = \young(3,*)\ .
\]
Since there is no $2$-segment, we have $\seg_2(b') = 0$. There is, however, one $3$-segment in the first row, so $\seg_3(b') = \seg(b') = 1$, and coefficient associated to $b'$ is $1-t^{-1}$. Using the BZL path, we have $\psi_\ii(b') = (\Circled{0};\Circled{1},1)$, so the contribution is $1-t^{-1}$.
\end{ex}

\begin{lemma}\label{lem:Aboxes}
Let $b\in \mathcal{T}(\infty)$ and $2 \le k \le r+1$. Suppose there are no $(k-1)$-segments in $b$. If $m$ is the maximal integer such that $\widetilde e_{k-1}^m b \neq 0$, then $m$ is the number of $k$-boxes comprising all $k$-segments of $b$.
\end{lemma}

\begin{proof}
By assumption, there are no $(k-1)$-segments in the tableau $b$. Thus, by the marginal largeness of the tableau $b\in \mathcal{T}(\infty)$, the $k$-signature (with all $(+,-)$-pairs removed) of $b$ has the form
\[
(-,-,\cdots,-,+).
\]
In particular, the sequence of $-$'s comes from $k$-segments in $b$, while $+$ comes from the $k-1$ in the mandatory $(k-1)$st row of $b$. Suppose there are $m$ such minus signs. By the definition of the signature rule, we have $\widetilde e_{k-1}^mb \in \mathcal{T}(\infty)$ but $\widetilde e_{k-1}^{m+1}b =0$. The claim is proved.
\end{proof} 

\begin{lemma}\label{lem:Apath}
Let $b\in \mathcal{T}(\infty)$ and  $\psi_\ii(b)=(a_{1,1}; a_{2,1}, a_{2,2}; \dots; a_{r,1}, \dots, a_{r,r})$. Suppose that $2 \le k \le r+1$. Then the sequence of operators
\[
\widetilde e_1^{a_{k-1,k-1}}\widetilde e_2^{a_{k-1,k-2}} \cdots \widetilde e_{k-1}^{a_{k-1,1}} \cdots \widetilde e_1^{a_{2,2}} \widetilde e_2^{a_{2,1}} \widetilde e_1^{a_{1,1}}
\] 
removes any and all $j$-segments from $b$, with $2 \le j \le k$. 
\end{lemma}

\begin{proof}
We proceed by induction on $k$.  First we assume that $k=2$. Notice that there is at most one $2$-segment in $b$, and it must occur in the first row. It is obvious from the definition of $\psi_\ii$ that $\widetilde e_1^{a_{1,1}}$ removes this $2$-segment.

Now suppose that for some $k \ge 2$, we have applied the sequence of operators to $b$:
\[
\widetilde e_1^{a_{k-2,k-2}} 
\cdots
\widetilde e_{k-2}^{a_{k-2,1}} 
\cdots
\widetilde e_1^{a_{2,2}} 
\widetilde e_2^{a_{2,1}} 
\widetilde e_1^{a_{1,1}} .
\]
Then, by the induction hypothesis, all $j$-segments for $2 \le j \le k-1$ are removed and we denote by $b'$ the resulting tableau. We apply to the tableau $b'$ the product of operators
\[
\widetilde e_1^{a_{k-1,k-1}}\widetilde e_2^{a_{k-1,k-2}} \cdots \widetilde
e_{k-1}^{a_{k-1,1}}.
\]
Since there are no $(k-1)$-segments in the tableaux $b'$, applying $\widetilde e_{k-1}^{a_{k-1,1}}$ will take any $k$-segment in the $(k-1)$st row completely out of the tableau and will take any $k$-segment not in the $(k-1)$st row to a $(k-1)$-segment by Lemma \ref{lem:Aboxes}. Applying this same argument to each of $\widetilde e_{k-2}^{a_{k-1,2}},\dots, \widetilde e_1^{a_{k-1,k-1}}$ consecutively, we prove the assertion of the lemma.
\end{proof}

\begin{proof}[Proof of Theorem \ref{thm:main-A}.]
Let $b \in \mathcal{T}(\infty)$. We will use the same notation $b$ to denote the
corresponding element in $\BB(\infty)$. In order to prove the theorem, we have only to show that $\seg(b)=\nc(\psi_\ii(b))$ by Corollary \ref{cor:BN}; that is, we need only to show that the number of all segments in $b$ is equal to the number of uncircled entries in $\psi_\ii(b)$. Recall that we may write $\psi_\ii(b)$ in a triangular array (\ref{eq:triangles}). Let $\nc_k(\psi_\ii(b))$ be the number of uncircled entries in the $k$th row of the triangular array. We will prove $\seg_k(b)=\nc_{k-1}(\psi_\ii(b))$ for each $k$. Then it will follow that $\seg(b)=\nc(\psi_\ii(b))$.

We first consider $2$-segments.  By definition, the $\widetilde e_1$ operator changes a $2$-box to a $1$-box. However, the only $2$-boxes that this will affect are boxes in a $2$-segment in the first row of $b$. With this observation, we apply $\widetilde e_1^{a_{1,1}}$ to the tableau $b$ where $a_{1,1}$ is the length of the $2$-segment. If $a_{1,1} =0$ (i.e., there is no $2$-segment), then we obtain a circle, but if $a_{1,1} >0$, then there is a $2$-segment and we do not get a circle. In both cases, $\seg_2(b)=\nc_1(\psi_\ii(b))$.

Now consider any $k$-segments in $b$, for $2 < k \le r+1$. By definition, any $k$-segment must appear between the first and $(k-1)$st row. Note that eliminating a $k$-segment from the tableaux $b$ occurs in the $(k-1)$st row of $\psi_\ii(b)$ according to Lemma \ref{lem:Apath}. If there are no $k$-segments anywhere in $b$, then the $(k-1)$st row of $\psi_\ii(b)$ will consist of $k-1$ zeros, each of which is circled. Thus $\seg_k(b) = 0$ (because there are no $k$-segments), which is exactly the number of uncircled entries in the $(k-1)$st row. On the other hand, if there is a $k$-segment in the $j$th row, where $1 \le j \le k-1$, then we have $\seg_k(b) \ge 1$. In particular, suppose that the length of the $k$-segment in the $j$th row is $m_1$. If there are no other $k$-segments anywhere in $b$, then $\widetilde e_{j}^{m_1} \cdots \widetilde e_{k-1}^{m_1}$ removes this $k$-segment entirely by Lemma \ref{lem:Apath}, so the $(k-1)$st row of $\psi_\ii(b)$ has the form
\[
(\underbrace{m_1,\dots,m_1}_{k-j\,\mathrm{times}},0,\dots,0).
\]
The only entry which is not circled is the last $m_1$, so $\seg_k(b)$ is again exactly the number of uncircled entries in that row; i.e., $\seg_k(b)=1=\nc_{k-1}(\psi_\ii(b))$.

Assume that there exists another $k$-segment in some row between $2$ and $j$, say in row $2 \le \ell < j$. Suppose that the $k$-segment in the $\ell$th row has length $m_2$. Then the $(k-1)$st row of $\psi_\ii(b)$ has the form
\[
(\underbrace{m_1+m_2,\dots,m_1+m_2}_{k-j
\,\mathrm{times}},\underbrace{m_2,\dots,m_2}_{j-\ell\,\mathrm{times}},0,\dots,0).
\]
In this case, $\seg_k(b) = 2$ and there are two uncircled entries in this row, so they match.

Continuing this way shows that $\seg_k(b)=\nc_{k-1}(\psi_\ii(b))$, which concludes the proof.
\end{proof}

From the above proof, we have obtained an interpretation of the string parametrization into information about the corresponding tableau:
\begin{cor}
Let $b\in \mathcal{T}(\infty)$ and  $\psi_\ii(b)=(a_{1,1}; a_{2,1}, a_{2,2}; \dots; a_{r,1}, \dots, a_{r,r})$. Then $a_{i,j}$ is the sum of lengths of $(i+1)$-segments in rows $1$ through $i-j+1$ of the tableau $b$.
\end{cor}

The following corollary will play an important role in the next section.

\begin{cor}\label{cor:seg-nz}
Let $b \in \mathcal{T}(\infty)$, and we denote by the same notation $b$ for the corresponding elements in $\BB(\infty)$ and $\CB$. Then we obtain
\[ \seg(b) = \nz(\phi_\ii(b)) = \nc(\psi_\ii(b)) . \]  
\end{cor}

\begin{proof}
In the proof of Theorem \ref{thm:main-A}, we showed $\seg(b)=\nc(\psi_\ii(b))$. By Corollary \ref{cor:BN}, we have $\nz(\phi_\ii(b)) = \nc(\psi_\ii(b))$.
\end{proof}

\vskip 0.8 cm
%%%%%%%%%%%%%%%%%%%%%
%%%%%%%%%%%%%%%%%%%%%
%%%%%%%%%%%%%%%%%%%%%
\section{Connections to MV polytopes and quiver varieties}

In this section, we investigate connections of our results to other realizations of crystals. In particular, we will interpret the meaning of segments of a tableau into the MV polytope model and Kashiwara-Saito's geometric construction of crystals, respectively. In the beginning of each of the following subsections, we briefly review the theory of MV polytopes and geometric construction of crystals. We refer the reader to the papers \cite{Kam:07,Kam:10} and \cite{KS:97} for more details.

\subsection{MV polytopes}

We require the Bruhat order $\ge$ on the Weyl group. We recall that there is an order on $P^\vee$, which we will also denote by $\ge$, defined by $\mu \ge \nu$ if and only if $\mu - \nu$ is a sum of positive coroots. We will also need a twisted partial order $\ge_w$ ($w\in W$) on $P^\vee$ such that $\mu\ge_w\nu$ if and only if $w^{-1}\cdot\mu \ge w^{-1}\cdot\nu$. We let $\Gamma = \{ w\cdot\Lambda_i : w\in W,\ i\in I \}$.

Let $M_\bullet = (M_\gamma)_{\gamma\in\Gamma}$ be a collection of integers. We say that $M_\bullet$ satisfies the {\it edge inequalities} if
\begin{equation}\label{edge}
M_{w\cdot\Lambda_i} + M_{w\sigma_i\cdot\Lambda_i} - M_{w\cdot\Lambda_{i-1}} -
M_{w\cdot\Lambda_{i+1}} \le 0
\end{equation}
for all $i \in I$ and $w\in W$, where we understand $M_{w\cdot\Lambda_{-1}}=0$ and $M_{w\cdot\Lambda_{r+1}}=0$. From such a collection, we define the {\it pseudo-Weyl polytope}
\[
\pwp(M_\bullet) = \{ \alpha \in \mathfrak{h}_\RR : \langle \alpha,\gamma \rangle \ge M_\gamma \text{ for all } \gamma \in \Gamma \}.
\]
Associated to such a pseudo-Weyl polytope is a map $w\mapsto \mu_w$ defined by the equation
\[
\langle \mu_w , w\cdot\Lambda_i \rangle = M_{w\cdot\Lambda_i}.
\]
The coweights $\mu_w$ should be regarded as vertices of the pseudo-Weyl polytope, and the collection $(\mu_w)_{w\in W}$ is called the {\it GGMS} datum of the pseudo-Weyl polytope.

Let $w\in W$ and $i,j\in I$ be such that $w\sigma_i > w$, $w\sigma_j > w$, and $i\neq j$. Define a sequence $M_\bullet = (M_\gamma)_{\gamma\in\Gamma}$ to satisfy the {\it tropical Pl\"ucker relation} at $(w,i,j)$ provided $|i-j|>1$, or if $|i-j|=1$ and
\[
M_{w\sigma_i\cdot\Lambda_i} + M_{w\sigma_j\cdot\Lambda_j} = 
\min(M_{w\cdot\Lambda_i}+M_{w\sigma_i\sigma_j\cdot\Lambda_j},\ 
M_{w\sigma_j\sigma_i\cdot\Lambda_i} + M_{w\cdot\Lambda_j}).
\]
We say $M_\bullet$ satisfies the tropical Pl\"ucker relations if it satisfies the tropical Pl\"ucker relations at each $(w,i,j) \in W\times I^2$. Finally, a collection $M_\bullet = (M_\gamma)_{\gamma\in \Gamma}$ is called a {\it BZ datum} of coweight $(\mu_1,\mu_2)$ if the following hold:
\begin{enumerate}
\item $M_\bullet$ satisfies the tropical Pl\"ucker relations.
\item $M_\bullet$ satisfies the edge inequalities.
\item If $\mu_\bullet = (\mu_w)_{w\in W}$ is the GGMS datum of $\pwp(M_\bullet)$, then $\mu_e = \mu_1$ and $\mu_{w_\circ} = \mu_2$, where $e$ is the identity element of $W$.
\end{enumerate}

\begin{dfn}
If $M_\bullet$ is a BZ datum of coweight $(\mu_1,\mu_2)$, then the corresponding pseudo-Weyl polytope $\pwp(M_\bullet)$ is called an {\it MV polytope} of coweight $(\mu_1,\mu_2)$.
\end{dfn}

For $\nu \in P^\vee$ and a BZ datum $M_\bullet$ of coweight $(\mu_1,\mu_2)$, we define  $\nu + \pwp(M_\bullet) = \pwp(M_\bullet')$, where $M_\gamma' = M_\gamma + \langle\nu,\gamma\rangle$ for each $\gamma \in \Gamma$ and $M_\bullet'$ is a BZ datum of coweight $(\mu_1+\nu,\mu_2+\nu)$. This yields an action of $P^\vee$ on the set of BZ datum, and hence on the set of MV polytopes. The orbit of an MV polytope of coweight $(\mu_1,\mu_2)$ under this action is called a {\it stable MV polytope} of coweight $\mu_1-\mu_2$. Note that, for each stable MV polytope of weight $\mu$, we may choose the unique representative of coweight $(\nu+\mu,\nu)$. Denote the set of all stable MV polytopes by $\mathcal{MV}$.

Assume that $\ii = (i_1,\dots,i_N) \in R(w_\circ)$ is an arbitrary long word. We set $w_k^\ii = \sigma_{i_1}\cdots\sigma_{i_k}$ and $\beta_k^\ii = w_{k-1}^\ii\cdot\alpha_{i_k}^\vee$ for each $1 \le k \le N$. (Here, we understand $w_0^\ii = e$, the identity element of the Weyl group.)  The reduced word $\ii$ determines a path in $\pwp(M_\bullet)$ given by the consecutive vertices $\mu_e,\mu_{\sigma_{i_1}},\dots,\mu_{w_\circ}$, and we obtain the vector $L_\ii(\pwp(M_\bullet)) = (c_1,\dots,c_N)$ consisting of the lengths of the edges along the $\ii$-path in $\pwp(M_\bullet)$; that is,
\[
\mu_{w_k^\ii} - \mu_{w_{k-1}^\ii} = c_k \beta_k^\ii.
\]
Thus, each positive coroot $\beta_k^\ii$ determines a direction in the coweight lattice and $c_k$ gives the length of the $\beta_k^\ii$th leg in the polytope. 
The vector $L_\ii(\pwp(M_\bullet)) = (c_1,\dots,c_N)$ is called the {\it $\ii$-Lusztig datum} of the MV polytope $\pwp(M_\bullet)$.

\begin{prop}[\cite{Kam:10}]
For any $\ii\in R(w_\circ)$, there is a bijection between $\mathcal{MV}$ and $\ZZ_{\ge0}^N$ given by the $\ii$-Lusztig datum of an MV polytope.
\end{prop}

We have already discussed that the Lusztig parametrization of an element $b\in \CB$ is a bijection $\phi_\ii\colon \CB\longrightarrow \ZZ_{\ge0}^N$ for any $\ii\in R(w_\circ)$. Thus, for any $b\in \CB$, there is an associated MV polytope, which we denote by $\pwp(b)$, with $\ii$-Lusztig datum $\phi_\ii(b)$.

Since we are considering the root system of type $A$, we have an isomorphism $\eta\colon \mathfrak{h}_\RR \longrightarrow \mathfrak{h}_\RR^\vee$ given by $\eta(\alpha_i^\vee)=\alpha_i$ for $i \in I$. If $\mathfrak{P}$ is a stable MV polytope of coweight $\mu$, then we also say that $A$ is of weight $\eta(\mu)$ and write
$\wt(\mathfrak{P})=\eta(\mu)$. Kamnitzer proved the following.

\begin{thm}[\cite{Kam:10}]\label{thm:MVcan}
The map $b\mapsto \pwp(b)$ is a weight preserving bijection $\CB\longrightarrow \mathcal{MV}$ such that $\phi_\ii(b) = L_\ii(\pwp(b))$.
\end{thm}

Assume that $\mathfrak{P} \in \mathcal{MV}$. We define $\nz(L_\ii(\mathfrak{P}))$ to be the number of nonzero entries in the $\ii$-Lusztig datum $L_\ii(\mathfrak{P})$. We see from the definitions that $\nz(L_\ii(\mathfrak{P}))$ is nothing but the number of edges in the $\ii$-path of $\mathfrak{P}$. The next corollary is obtained from Corollary \ref{cor:seg-nz} and Theorem \ref{thm:MVcan}.

\begin{cor}\label{cor:segd''}
Let $b\in \mathcal{T}(\infty)$ and denote by the same notation $b$ the corresponding element in $\CB$. Then we have
\[
\seg(b) = \nz(L_\ii(\pwp(b))).
\]
\end{cor}

We are now ready to present the Gindikin-Karpelevich formula as a sum over MV polytopes. 

\begin{cor}
Let $\Phi$ be the root system of $\mathfrak{sl}_{r+1}$. Then for any $\ii\in R(w_\circ)$, we have
\[
\prod_{\alpha\in\Phi^+} \frac{1-t^{-1}\zz^\alpha}{1-\zz^\alpha} =
\sum_{\mathfrak{P}\in\mathcal{MV}} (1-t^{-1})^{\nz(L_\ii(\mathfrak{P}))}\zz^{-\wt(\mathfrak{P})}.
\]
\end{cor}

\begin{proof}
The corollary follows from Theorem \ref{thm:MVcan}, Corollary \ref{cor:segd''} and Theorem \ref{thm:main-A}.
\end{proof}

\subsection{Quiver varieties}

Let $I=\{ 1, \dots , r\}$ be the set of vertices and $H$ be the set of arrows such that $i \rightarrow j$ with $i-j=\pm 1$, $i,j \in I$. Then $(I, H)$ is the double quiver of type $A_r$:
\[
\begin{tikzpicture}
\matrix (m) [matrix of math nodes, row sep=.3in, column sep=.5in, text height=1.5ex, text depth=0.25ex]
 { 1 & 2 & \phantom{x}\cdots\phantom{x} & r.  \\ };
\path[<-]
 ([yshift=2]m-1-1.east) edge[bend left] ([yshift=2]m-1-2.west)
 ([yshift=-2]m-1-2.west) edge[bend left] ([yshift=-2]m-1-1.east)
 ([yshift=2]m-1-2.east) edge[bend left] ([yshift=2]m-1-3.west)
 ([yshift=-2]m-1-3.west) edge[bend left] ([yshift=-2]m-1-2.east)
 ([yshift=2]m-1-3.east) edge[bend left] ([yshift=2]m-1-4.west)
 ([yshift=-2]m-1-4.west) edge[bend left] ([yshift=-2]m-1-3.east);
\end{tikzpicture}
\]  
If $h\in H$ is the arrow $i \rightarrow j$, then we set $\out(h) = i$ and $\IN(h) = j$. We choose an orientation $\Omega \subset H$ of the quiver and its opposite $\overline \Omega$ so that we have
\begin{align}
(I,\Omega) &= 1 \longleftarrow 2 \longleftarrow \cdots \longleftarrow r,\label{eq:Iomega}\\
(I,\overline\Omega) &= 1 \longrightarrow 2 \longrightarrow \cdots \longrightarrow r.
\end{align}

Given an $I$-graded vector space $\VV= \bigoplus_{i=1}^r \VV_i$, we set 
\[
\ddim \VV = \sum_{i\in I} \dim (\VV_i)\, \alpha_i \in Q^+,
\]
where $Q^+ = \bigoplus_{i=1}^r \ZZ_{\ge 0}\,  \alpha_i$.  Now define $\CC$-vector spaces
\begin{align*}
\EE_\VV &:= \bigoplus_{h\in H}\operatorname{Hom}(\VV_{\out(h)},\VV_{\IN(h)}) , \\
\EE_{\VV,\Omega} &:= \bigoplus_{h \in \Omega} \operatorname{Hom}(\VV_{\out(h)},\VV_{\IN(h)}) .
\end{align*}
The group $G_\VV := \prod_{i\in I} \operatorname{GL}(\VV_i)$ acts on both $\EE_\VV$ and $\EE_{\VV,\Omega}$ by
\[
(g,x) = \big((g_i), (x_h)\big) \mapsto \big(g_{\IN(h)} \, x_h \, g_{\out(h)}^{-1}\big).\]
Let $\omega$ be the nondegenerate, $G_\VV$-invariant, symplectic form on $\EE_\VV$
defined by
\[
\omega(x,y) := \sum_{h\in H} \epsilon(h)\operatorname{Tr}(x_hy_{\overline h}),
\]
where $\epsilon\colon H \longrightarrow \{\pm1\}$ is the function such that $\epsilon(h) = 1$ if $h\in \Omega$ and $\epsilon(h) = -1$ if $h \in \overline\Omega$. Let $\mathfrak{gl}_\VV = \bigoplus_{i\in I} \operatorname{End}(\VV_i)$ be the Lie algebra of $G_\VV$, which acts on $\EE_\VV$ via $A \cdot x = [A,x]$, for $A\in\mathfrak{gl}_\VV$ and $x \in \EE_\VV$. Let $\mu\colon \EE_\VV \longrightarrow \mathfrak{gl}_\VV$ be the moment map associated with the $G_\VV$-action on $\EE_\VV$, whose $i$-th component $\mu_i\colon \EE_\VV \longrightarrow \operatorname{End}(\VV_i)$ is given by
\[
\mu_i(x) = \sum_{\substack{h\in H \\ i = \IN(h)}} \epsilon(h)x_hx_{\overline h}.
\]

Finally, we define 
\[
\Lambda_\VV = \{ x \in \EE_\VV : \mu(x) = 0  \}.
\] 
The variety $\Lambda_\VV$ is called the {\em Lusztig quiver variety}.
For $\alpha=\sum_{i=1}^r k_i \alpha_i \in Q^+$, let $\VV(\alpha)= \bigoplus_{i=1}^r \VV_i(\alpha)$ be an $I$-graded vector space with $\ddim \VV(\alpha)=\alpha$. Let $\operatorname{Irr}\Lambda(\alpha)$ denote the set of irreducible components of $\Lambda(\alpha):= \Lambda_{\VV(\alpha)}$ and define
\[
\mathfrak{X}(\infty) = \bigsqcup_{\alpha \in Q^+} \operatorname{Irr} \Lambda(\alpha).
\]
Kashiwara and Saito defined a crystal structure on $\mathfrak{X}(\infty)$ and showed the following theorem.

\begin{thm}[{\cite{KS:97}}]
There is a crystal isomorphism $\mathfrak{X}(\infty) \cong \BB(\infty)$.
\end{thm}

For $k,\ell \in \ZZ$ such that $1 \le k \le \ell \le r$, we define $\big(\VV(k,\ell),x(k,\ell)\big)$ to be the representation of $(I, \Omega)$ with $\VV(k,\ell)_i = \CC$ for $k \le i \le \ell$ and $\VV(k,\ell)_i = 0$ otherwise. The maps $x(k,\ell)$ between the nonzero vector spaces are the identity and zero otherwise. The representation $\big(\VV(k,\ell),x(k,\ell)\big)$ is indecomposable, and any indecomposable finite-dimensional representation $(\VV,x)$ of $(I, \Omega)$ is isomorphic to some $\big(\VV(k,\ell),x(k,\ell)\big)$.

%\begin{ex}
%If $r=5$.  Then $\big(\VV(2,4),x(2,4)\big)$ can be visualized as 
%\[
%0 \overset{0}\longleftarrow \CC \overset{1}\longleftarrow \CC \overset{1}\longleftarrow \CC \overset{0}\longleftarrow 0.
%\]
%\end{ex}

The following proposition is well-known.

\begin{prop}[\cite{Luszt:91}]
Let $(I,\Omega)$ be the quiver of type $A_r$ from \eqref{eq:Iomega}. The irreducible components of $\Lambda_\VV$ are the closures of conormal bundles of the $G_\VV$-orbits in $\EE_{\VV,\Omega}$.
\end{prop}

Assume that $X \in \mathfrak{X}(\infty)$ is an irreducible component of $\Lambda_\VV$ for some $\VV=\VV(\alpha)$. Then there exists the corresponding $G_\VV$-orbit $\mathscr{O}$, which consists of all the representations of $(I, \Omega)$ that are isomorphic to a sum $\VV(X)$ of indecomposable representations $\big(\VV(k,\ell),x(k,\ell)\big)$. We define $\gamma(X)$ to be the number of different indecomposable representations (not counting multiplicity) in the sum $\VV(X)$. We also set $\wt(X)=\ddim \VV$. We obtain the following interpretation of Thoerem \ref{thm:main-A} in the framework of the quiver variety: 

\begin{cor}
Let $\Phi$ be the root system of $\mathfrak{sl}_{r+1}$.  Then
\begin{equation}
\prod_{\alpha\in\Phi^+} \frac{1-t^{-1}\bm z^{\alpha}}{1-\bm z^{\alpha}} = \sum_{X\in \mathfrak{X}(\infty)} (1-t^{-1})^{\gamma(X)} \bm z^{\wt(X)}.
\end{equation}
\end{cor}


\begin{thebibliography}{10}

\bibitem{BZ:96}
A.~Berenstein and A.~Zelevinsky, \emph{Canonical bases for the quantum group of
  type {$A_r$} and piecewise-linear combinatorics}, Duke Math. J. \textbf{82}
  (1996), no.~3, 473--502. \MR{1387682 (97g:17007)}

\bibitem{BZ:01}
\bysame, \emph{Tensor product multiplicities, canonical bases, and totally
  positive varieties}, Invent. Math. \textbf{143} (2001), 77--128.

\bibitem{BBF:11}
B.~Brubaker, D.~Bump, and S.~Friedberg, \emph{{W}eyl group multiple {D}irichlet
  series: Type {$A$} combinatorial theory}, Annals of Mathematics Studies, vol.
  AM-175, Princeton Univ. Press, New Jersey, 2011.

\bibitem{BN:10}
D.~Bump and M.~Nakasuji, \emph{Integration on {$p$}-adic groups and crystal
  bases}, Proc. Amer. Math. Soc. \textbf{138} (2010), no.~5, 1595--1605.
  
\bibitem{GL:93}
I.~Grojnowski and G.~Lusztig, \emph{A comparison of bases of quantized
  enveloping algebras}, Linear algebraic groups and their representations
  ({L}os {A}ngeles, {CA}, 1992), Contemp. Math., vol. 153, Amer. Math. Soc.,
  Providence, RI, 1993, pp.~11--19. \MR{1247495 (94m:17012)}

\bibitem{HK:02}
J.~Hong and S.-J. Kang, \emph{Introduction to quantum groups and crystal
  bases}, Graduate Studies in Mathematics, vol.~42, American Mathematical
  Society, Providence, RI, 2002. \MR{1881971 (2002m:17012)}

\bibitem{HL:08}
J.~Hong and H.~Lee, \emph{{Y}oung tableaux and crystal {$\mathcal{B}(\infty)$}
  for finite simple {L}ie algebras}, J. Algebra \textbf{320} (2008),
  3680--3693.

\bibitem{Kam:07}
J.~Kamnitzer, \emph{The crystal structure on the set of {M}irkovi\'c-{V}ilonen
  polytopes}, Adv. Math. \textbf{215} (2007), no.~1, 66--93. \MR{2354986
  (2009a:17021)}

\bibitem{Kam:10}
\bysame, \emph{Mirkovi\'c-{V}ilonen cycles and polytopes}, Ann. of Math. (2)
  \textbf{171} (2010), no.~1, 245--294. \MR{2630039}

\bibitem{Kang:03}
S.-J. Kang, \emph{Crystal bases for quantum affine algebras and combinatorics
  of {Y}oung walls}, Proc. London Math. Soc. (3) \textbf{86} (2003), no.~1,
  29--69. \MR{1971463 (2004c:17028)}

\bibitem{Kash:91}
M.~Kashiwara, \emph{On crystal bases of the $q$-analogue of universal
  enveloping algebras}, Duke Math. J. \textbf{63} (1991), no.~2, 465--516.

\bibitem{KN:94}
M.~Kashiwara and T.~Nakashima, \emph{Crystal graphs for representations of the
  $q$-analogue of classical {L}ie algebras}, J. Algebra \textbf{165} (1994),
  295--345.

\bibitem{KS:97}
M.~Kashiwara and Y.~Saito, \emph{Geometric construction of crystal bases}, Duke
  Math. J. \textbf{89} (1997), no.~1, 9--36. \MR{1458969 (99e:17025)}

\bibitem{KL:11}
H.~H. Kim and K.-H. Lee, \emph{Representation theory of {$p$}-adic groups and
  canonical bases}, Adv.\ Math. \textbf{227} (2011), no.~2, 945--961.

\bibitem{LS2:11}
K.-H. Lee and B.~Salisbury, \emph{A combinatorial description of the
  {G}indikin-{K}arpelevich formula in types $B, C, D , G$}, in preparation.

\bibitem{Lit:95}
P.~Littelmann, \emph{Paths and root operators in representation theory}, Ann.
  of Math. (2) \textbf{142} (1995), no.~3, 499--525. \MR{1356780 (96m:17011)}

\bibitem{Lit:98}
\bysame \emph{Cones, crystals, and patterns}, Transform. Groups {\bf 3} (1998), no. 2, 145--179.


\bibitem{Luszt:90}
G.~Lusztig, \emph{Canonical bases arising from quantized enveloping algebras},
  J. Amer. Math. Soc. \textbf{3} (1990), no.~2, 447--498. \MR{1035415
  (90m:17023)}

\bibitem{Luszt:91}
\bysame, \emph{Quivers, perverse sheaves, and quantized enveloping algebras},
  J. Amer. Math. Soc. \textbf{4} (1991), no.~2, 365--421. \MR{1088333
  (91m:17018)}

\bibitem{Luszt:93}
\bysame, \emph{Introduction to quantum groups}, Progress in Mathematics, vol.
  110, Birkh\"auser Boston Inc., Boston, MA, 1993. \MR{1227098 (94m:17016)}

\bibitem{McN:11}
P.~J. McNamara, \emph{Metaplectic {W}hittaker functions and crystal bases},
  Duke Math. J. \textbf{156} (2011), no.~1, 29--31. \MR{2746386}

\bibitem{combinat}
The {S}age-{C}ombinat community, \emph{{S}age-{C}ombinat: enhancing {S}age as a
  toolbox for computer exploration in algebraic combinatorics}, 2008, {\tt
  http://combinat.sagemath.org}.

\bibitem{sage}
W.\thinspace{}A. Stein et~al., \emph{{S}age {M}athematics {S}oftware ({V}ersion
  4.6.1)}, The Sage Development Team, 2011, {\tt http://www.sagemath.org}.

\end{thebibliography}
\end{document}